\documentclass[12pt]{amsart}
\numberwithin{equation}{section}
\newtheorem{theorem}{Theorem}[section]
\newtheorem{lemma}[theorem]{Lemma}

\newtheorem{conjecture}[theorem]{Conjecture}

\theoremstyle{definition}
\newtheorem{definition}[theorem]{Definition}

\newtheorem{question}[theorem]{Question}
\theoremstyle{remark}

\usepackage{amssymb}
\usepackage{caption}
\usepackage{fullpage}
\usepackage{amsmath}
\usepackage{amsfonts}
\usepackage{latexsym}
\usepackage{hyperref}

\title{Character Values of Stanley Sequences}
\author{Mehtaab Sawhney}
\thanks{Massachusetts Institute of Technology, Cambridge MA. Email: \texttt{msawhney98@yahoo.com}}

\begin{document}

\begin{abstract}
Stanley and Odlyzko proposed a method for greedily constructing sets with no $3$-term arithmetic progressions. It is conjectured that there is a dichotomy between such sequences: those that have a periodic structure as the sequence satisfies certain recurrence relations while others appear to be chaotic. One large class of sequences that have these periodic behaviors are known as independent sequences that have two parameters, a character and a growth factor. It was conjectured by Rolnick that all but a finite set of integers can be achieved as characters of a independent sequences. Previously the only large class of integers known to be characters where those with base $3$ representations consisting solely of the digits $0$ and $2$. This paper dramatically improves this result by demonstrating that all even integers not congruent to $244\mod 486$ can be achieved as characters, therefore demonstrating that the set of all characters has a positive lower density. 
\end{abstract}

\maketitle
\section{Introduction}
A set of nonnegative integers is defined to be $p$-free if it contains no arithmetic progressions of length $p$, where $p$ is a prime. Define $r_p(n)$ to be the cardinality of the largest subset of $\{0,1,\ldots,n\}$  that contains no $p$-term arithmetic progressions.  Erd\H{o}s and Tur\'an \cite{ET} stated a conjecture of Szekeres, that for an odd prime $p$, that $r_p(n)=\Theta(n^{\log_{p-1}{p}})$. However this conjecture was disproved, and the best known lower bound for $p=3$ is $n^{1-o(1)}$ due to Elkin \cite{E}. The best upper bound for $p=3$ is $O\big(\frac{n(\log\log n)^5}{\log n}\big)$ due to a  result of Sanders \cite{S}. \\
In 1978 Odlyzko and Stanley \cite{OS} proposed the following method for constructing $3$-free arithmetic progressions.
\begin{definition}
Let $A:=\{a_1,\ldots,a_n\}$ be a finite set of nonnegative integers with no nontrivial $3$-term arithmetic progressions that contains $0$ and $a_1<a_2<\cdots<a_n$. For each integer $k\ge n$, let $a_{k+1}$ be the least integer greater than $a_k$ such that $\{a_1,\ldots, a_k,a_{k+1}\}$ has no arithmetic progressions. The Stanley sequences $S(A)=S(a_1,\ldots, a_n)$ is the sequence $a_1,\ldots, a_n, a_{n+1}, \ldots$.
\end{definition}
The simplest example of Stanley sequence is $S(0)$, which is the set of integers which have no $2$'s in there ternary expansions. Odlyzko and Stanley \cite{OS} demonstrate that $S(0, 3^k)$ and $S(0, 2\cdot 3^k)$ have similarly explicit descriptions, and Rolnick \cite{R} gives a vast generalization of this result. However, Erd\H{o}s and Graham \cite{EG} note that a sequences as simple as $S(0,4)$ appears to have chaotic behavior. This leads to the following conjecture which was suggested by Odlyzko and Stanley \cite{OS}.
\begin{conjecture}
Let $S(A)=(a_n)$ be a Stanley sequence. Then asymptotically, one the following two growth rates is satisfied.\\
\begin{itemize}
\item (Type 1) $\frac{\alpha}{2}\le \lim\inf a_n/n^{\log_2{3}}\le \lim\sup a_n/n^{\log_2{3}}\le \alpha$
\item (Type 2) $a_n=O\bigg( \frac{n^2}{\log n} \bigg )$
\end{itemize}
\end{conjecture}
Although Odlyzko and Stanley had only considered Type 1 sequences when $\alpha=1$, Rolnick and Venkataramana \cite{RV} demonstrated that all rational $\alpha$ with denominators which are powers of three can be achieved. Furthermore, although Odlyzko and Stanley \cite{OS} provide ``probabilistic" justification that a ``random" Stanley sequence should be of Type 2, no sequence is known to have Type 2 growth. Lindhurst \cite{L} computes $S(0,4)$ for large values and observes that it appears to follows Type 2 growth. Rolnick \cite{R} began a study of attempting to classify all Type 1 sequences and noted that a large class of sequences of Type 1 have the following structure.
\begin{definition}
A Stanley sequence $S(A)=(a_n)$ is \emph{independent} if there are constants $\lambda$ and $\chi$ such that for all $k\ge \chi$ and $0\le i<2^k$ 
\begin{itemize}
\item $a_{2^k+i}=a_{2^k}+a_i$
\item $a_{2^k}=2a_{2^k-1}-\lambda+1$
\end{itemize}
\end{definition}
The constant $\lambda$ is referred to as the \emph{character} of the Stanley sequence while $a_{2^\chi}$ is referred to as the \emph{repeat factor} when $\chi$ is chosen minimally. In particular, in a vague sense, $2^{\chi}$ is first point at which the Stanley sequence appears to demonstrate periodic behavior while $\lambda$ is the shift after every block of terms. Rolnick and Venkatarama \cite{RV} demonstrated that all sufficiently large integers can be repeat factors of independent sequences. For the character on the other hand, Rolnick \cite{R} demonstrated that $\lambda$ is nonnegative, and conjectured that all but a finite set of character values are possible, and gave a set of excluded integers based on computations for character values between $0$ and $74$.
\begin{conjecture}\cite{R}
All nonnegative integers $\lambda$ not in the set $\{1,3,5,9,11,15\}$ can be achieved as characters of an independent Stanley sequence.
\end{conjecture}
In private communication, Moy \cite{moy} has demonstrated that $\{1,3,5,9,11,15\}$ cannot be achieved as characters of independent Stanley sequences. For infinite sets of integers whose elements can be achieved as characters of independent Stanley sequences, Rolnick \cite{R} had previously demonstrated that all integers with base $3$ representations consisting of the integers $0$ and $2$ have the form. In this paper we generalize the previous result as follows.
\begin{theorem}
All nonnegative integers $\lambda$ $\equiv 0$$\mod 2$ and $\lambda\not\equiv 244 \mod 486$ can be achieved as characters of independent Stanley sequences. 
\end{theorem}
In order to prove the main theorem we need several notions related to modular Stanley sets and basic sequences defined by Moy and Rolnick \cite{RM}. 
\begin{definition}
Suppose a set $A\subseteq \{0,1,\ldots, N-1\}$, $0\in A$, and for every integer $x$ there exists $y\ge z$ in $A$ such that $x\equiv 2y-z \mod N$. Define $A$ to be a \emph{modular set}$\mod N$ if $x,y,z\in A$ and $x\equiv 2y-z \mod N$ implies $x=y=z$. Essentially, the first condition is that every integer is covered by a 3-term arithmetic progression$\mod N$ while the second condition states that the set itself has no 3-term arithmetic progressions$\mod N$ other than the trivial arithmetic progressions.
\end{definition}
The key use of modular sequences in this paper stems from Theorem 2.4 in Moy and Rolnick \cite{RM} which is used when considering basic sequences. Note that $X+Y=\{x+y~|~x\in X, y\in Y\}$ and $\alpha X=\{\alpha\cdot x~|~x\in X\}$ for sets $X$ and $Y$.
\begin{theorem}\cite{RM}
Suppose that $A$ is a modular set modulo $N$ with N being a positive integer. Then $S(A)=A+N\cdot S(0)$.
\end{theorem}
Next, we introduce the notion of basic sequences defined by Moy and Rolnick \cite{RM} which generalize the notion of writing the terms of the Stanley sequence in base 3 that is present in $S(0)$.
\begin{definition}
We say that a Stanley sequences $(a_n)$ is \emph{basic} if there exists a basis $B=(b_k)$ such that $b_k=\alpha\cdot 3^k$ for $k$ sufficiently large and the elements of $A$ are sums of finite subsets of $B$. In particular \[\bigg\{\sum_{k=0}^{\infty}\delta_kb_k~|~\delta_k\in\{0,1\}, \sum_{k=0}^{\infty}\delta_k<\infty\bigg\}\] are the elements of a certain Stanley sequences. The sequence $B$ is referred to as the \emph{basis} of this Stanley sequence.
\end{definition}
The easiest example of a basic sequence is $S(0)$, which is basic with basis $B=\{1,3,3^2,\ldots \}$ by the explicit description given earlier. Finally we introduce a sufficient condition for a sequence to be basis, this is Theorem 4.4 in Moy and Rolnick \cite{RM}.
\begin{theorem}\cite{RM}
The sequence $B=(b_k)$ is a valid basis provided that 
\begin{itemize}
\item $3^k$ is the largest power of $3$ that divides $b_k$.
\item $b_k=3^k$ for $k$ sufficiently large.
\end{itemize}
\end{theorem}
\section{Characters of Stanley Sequences}
The first lemma is a technical result that allows us to extend the notion of ``basic" sequences by removing a set of basis elements and replacing them with a set of elements that is closely related 
to a modular set. This can be viewed as combining the first set of digits and then separately dealing with the basis elements, and given this, perspective the proof closely follows that of Theorem 4.4 in Moy and Rolnick \cite{RM}. For technical reasons, however, it is first necessary to introduce the notion of near-modular sets, which generalize the notion of modular sets in an obvious way.
\begin{definition}
A \emph{near-modular} set $A$ with respect to a positive integer $N$ is a set of nonnegative integers satisfying 
\begin{itemize}
\item $0\in A$ 
\item If integers $x$, $y$, $z$ in $A$ satisfy $x\equiv 2y-z\mod N$ then $x=y=z$.
\item For any integer $x$ there exist $y$ and $z$ in $A$ such that $x\equiv 2y-z\mod N$ and $y\ge z$.
\end{itemize}
\end{definition}
Note that unlike modular sets defined in by Moy and Rolnick in \cite{RM}, it is possible for near modular sets to have $\max\{A\}\ge N$.
\begin{lemma}
Suppose that $A$ is a near-modular set with respect to $3^{\ell}$ for $\ell\ge 1$ and $|A|=2^{\ell}$. Furthermore, consider a sequence $(b_k)$ for $k\ge 0$ such that the following two conditions hold: 
\begin{itemize}
\item $3^{k+\ell}$ is the largest power of three dividing $b_k$.
\item For $k$ sufficiently large, $b_k=3^{k+\ell}$.
\end{itemize}
Then
\[\bigg\{a+\sum_{k=0}^{\infty} \delta_kb_k~|~a\in A, \delta_k\in\{0,1\}, \sum_{k=0}^{\infty}\delta_k<\infty\bigg\}\]
are the elements of an independent Stanley sequence.
\end{lemma}
\begin{proof}
There exists an index $n_0$ such that for all $n\ge n_0$, $b_n=3^{n+\ell}$, and $b_n> \max{A}+\sum_{k=0}^{n-1} b_k$. Using Theorem 1.7 it suffices to demonstrate that 
\[L:=\bigg\{a+\sum_{k=0}^{n_0-1} \delta_kb_k~|~ a\in A, \delta_k\in\{0,1\}\bigg\}\] is a modular set modulo $b_{n_0}=3^{n_0+\ell}$.\\\\
In order to prove that $L$ is modular first note that $0\in L$ and $L\subseteq \{0,\ldots,3^{n_0+\ell}-1\}$. Now suppose that $x$, $y$, and $z$ in $L$ satisfy $x\equiv 2y-z \mod 3^{n_0+\ell}$ with $x$, $y$, and $z$ not all equal. Let $x=x_A+x_B$, $y=y_A+y_B$, and $z=z_A+z_B$, where $x_A, y_A, z_A$ are the parts of $x$, $y$, and $z$ from $A$ and the remaining parts $x_B, y_B, z_B$ come from $B=(b_k)$. Then for a particular number $x$, define the $\delta_i(x)$ as the coefficient of $b_i$ when considered as an element in $L$. In order to see that $\delta_i(x)$ is well define suppose that $x$ has two different representations. Then taking$\mod 3^{\ell}$, the component coming from $A$ is unique. Then taking$\mod 3^{{\ell}+1}$ it follows that $\delta_0(x)$ is well defined and continuing in such a manner it follows that $\delta_i(x)$ is well defined. Now since $3^{\ell}$ divides $x_B, y_B, z_B$ it follows that $x_A\equiv 2y_A-z_A \mod 3^{\ell}$. However $A$ is near-modular with respect to $3^{\ell}$ and thus it follows that $x_A=y_A=z_A$. Therefore $x_B\equiv 2y_B-z_B \mod 3^{k+\ell}$. Furthermore note that for $k\ge 1$, $3^{\ell+1}$ divides $b_k$ and since $x_B\equiv 2y_B-z_B \mod 3^{\ell+1}$ it follows that $\delta_0(x)=\delta_0(y)=\delta_0(z)=0$ or $\delta_0(x)=\delta_0(y)=\delta_0(z)=1$ since $\delta_0(t)\in\{0,1\}$ for all elements in $L$. Continuing in this manner, it follows that $x=y=z$ which is a contradiction to the assumption that all $x$, $y$, and $z$ are not all equal\\\\
Now consider $0\le x<3^{n_0+\ell}$. We construct $y$ and $z$ such that $2y-z\equiv x\mod 3^{n_0+\ell}$. By definition of near-modular there exist $y_A\ge z_A$ such that $2y_A-z_A\equiv x_A \mod 3^{\ell}$. Now set $y_A=y^{(-1)}$ and $z_A=z^{(-1)}$ and note that $2y^{(-1)}-z^{(-1)}$ and $x$ match for the first $\ell$ digits. (Digits refers to the ternary expansion.) We now inductively construct the desired integers so that $2y^{(m-1)}-z^{(m-1)}$ matches the first $\ell+m$ digits of $x$, $y^{(m-1)}\ge x^{(m-1)}$, and $\delta_i(y^{(m-1)})=\delta_i(z^{(m-1)})=0$ for $i\ge m$. \\
There are two possible cases with several possible sub-cases.\\\\
Case 1: Suppose that ${b_m}\equiv {3^{m+\ell}} \mod 3^{m+\ell+1}$. 
\begin{itemize}
\item Case 1a: We now subdivide cases based the $\ell+m+1^{st}$ digit of $x$. Suppose that $x$ has $\ell+m+1^{st}$ digit $0$. If $y^{(m-1)}$ and $z^{(m-1)}$ have $\ell+m+1^{st}$ digits $0$ and $0$, $2$ and $1$, or $1$ and $2$, set $y^{(m)}=y^{(m-1)}$and $z^{(m)}=z^{(m-1)}$. If $y^{(m-1)}$ and $z^{(m-1)}$ have $\ell+m+1^{st}$ digits $1$ and $0$, $0$ and $1$, or $2$ and $2$, then $y^{(m)}=b_m+y^{(m-1)}$ and $z^{(m)}=b_m+y^{(m-1)}$. Finally, if $y^{(m-1)}$ and $z^{(m-1)}$ have $\ell+m+1^{st}$ digits $2$ and $0$, $0$ and $2$, or $1$ and $1$, then set $y^{(m)}=b_m+y^{(m-1)}$ and $z^{(m)}=z^{(m-1)}$. 
\item Case 1b: Suppose that $x$ has $\ell+m+1^{st}$ digit $1$. If $y^{(m-1)}$ and $z^{(m-1)}$ have $\ell+m+1^{st}$ digits $0$ and $0$, $2$ and $1$, or $1$ and $2$, set $y^{(m)}=y^{(m-1)}+b_m$ and $z^{(m)}=z^{(m-1)}+b_m$. If $y^{(m-1)}$ and $z^{(m-1)}$ have $\ell+m+1^{st}$ digits $1$ and $0$, $0$ and $1$, or $2$ and $2$, then $y^{(m)}=b_m+y^{(m-1)}$ and  $z^{(m)}=z^{(m-1)}$. Finally, if $y^{(m-1)}$ and $z^{(m-1)}$ have $\ell+m+1^{st}$ digits $2$ and $0$, $0$ and $2$, or $1$ and $1$, then set $y^{(m)}=y^{(m-1)}$ and $z^{(m)}=z^{(m-1)}$. 
\item Case 1c: Finally suppose that $x$ has $\ell+m+1^{st}$ digit $2$. If $y^{(m-1)}$ and $z^{(m-1)}$ have $\ell+m+1^{st}$ digits $0$ and $0$, $2$ and $1$, or $1$ and $2$, then $y^{(m)}=y^{(m-1)}+b_m$ and $z^{(m)}=z^{(m-1)}$. If $y^{(m-1)}$ and $z^{(m-1)}$ have $\ell+m+1^{st}$ digits $1$ and $0$, $0$ and $1$, or $2$ and $2$, then set $y^{(m)}=y^{(m-1)}$ and $z^{(m)}=z^{(m-1)}$. Finally, if $y^{(m-1)}$ and $z^{(m-1)}$ have $\ell+m+1^{st}$ digits $2$ and $0$, $0$ and $2$, or $1$ and $1$, then $y^{(m)}=y^{(m-1)}+b_m$ and $z^{(m)}=z^{(m-1)}+b_m$. 
\end{itemize}
Case 2: Suppose that $b_m\equiv 2\cdot{3^{m+\ell}} \mod 3^{m+\ell+1}$. 
\begin{itemize}
\item Case 2a: Suppose that $x$ has $\ell+m+1^{st}$ digit $0$. If $y^{(m-1)}$ and $z^{(m-1)}$ have $\ell+m+1^{st}$ digits $0$ and $0$, $2$ and $1$, or $1$ and $2$, set $y^{(m)}=y^{(m-1)}$ and $z^{(m)}=z^{(m-1)}$. If $y^{(m-1)}$ and $z^{(m-1)}$ have $\ell+m+1^{st}$ digits $1$ and $0$, $0$ and $1$, or $2$ and $2$, then $y^{(m)}=y^{(m-1)}+b_m$ and $z^{(m)}=z^{(m-1)}$. Finally, if $y^{(m-1)}$ and $z^{(m-1)}$ have $\ell+m+1^{st}$ digits $2$ and $0$, $0$ and $2$, or $1$ and $1$, then $y^{(m)}=y^{(m-1)}+b_m$ and $z^{(m)}=z^{(m-1)}+b_m$. 
\item Case 2b: Suppose that $x$ has $\ell+m+1^{st}$ digit $1$. If $y^{(m-1)}$ and $z^{(m-1)}$ have $\ell+m+1^{st}$ digits $0$ and $0$, $2$ and $1$, or $1$ and $2$, then $y{(m)}=y^{(m-1)}+b_m$ and $z^{(m)}=z^{(m-1)}$. If $y^{(m-1)}$ and $z^{(m-1)}$ have $\ell+m+1^{st}$ digits $1$ and $0$, $0$ and $1$, or $2$ and $2$, then $y^{(m)}=y^{(m)}+b_m$ and $z^{(m)}=z^{(m-1)}+b_m$. Finally if $y^{(m-1)}$ and $z^{(m-1)}$ have $\ell+m+1^{st}$ digits $2$ and $0$, $0$ and $2$, or $1$ and $1$, then set $y^{(m)}=y^{(m-1)}$ and $z^{(m)}=z^{(m-1)}$. 
\item Case 2c: Finally suppose that $x$ has $\ell+m+1^{st}$ digit $2$. If $y^{(m-1)}$ and $z^{(m-1)}$ have $\ell+m+1^{st}$ digits $0$ and $0$, $2$ and $1$, or $1$ and $2$, add $y^{(m)}=y^{(m-1)}+b_m$ and $z^{(m)}=z^{(m-1)}+b_m$. If $y^{(m-1)}$ and $z^{(m-1)}$ have $\ell+m+1^{st}$ digits $1$ and $0$, $0$ and $1$, or $2$ and $2$, then set $y^{(m)}=y^{(m-1)}$ and $z^{(m)}=z^{(m-1)}$. Finally if $y^{(m-1)}$ and $z^{(m-1)}$ have $\ell+m+1^{st}$ digits $2$ and $0$, $0$ and $2$, or $1$ and $1$, then add $y^{(m)}=y^{(m-1)}+b_m$ and $z^{(m)}=z^{(m-1)}$. 
\end{itemize}
Inductively continuing until $m=n_0-1$, we have that $2y^{(n_0-1)}-z^{(n_0-1)}\equiv x \mod 3^{n_0+\ell}$ and thus taking $y^{(n_0-1)}$ and $z^{(n_0-1)}$ to be $y$ and $z$ gives $2y-z\equiv x \mod 3^{n_0+\ell}$ with $y\ge z$ and $y, z\in L$. Thus the desired result follows.\\
\end{proof}
With this particular lemma in hand it is possible to prove the main result, Theorem 1.5 mentioned earlier. For the initial case, we rely on Theorem 1.9 mentioned earlier and due to Moy and Rolnick \cite{RM}.
\begin{lemma}
All nonnegative integers $\lambda$ $\equiv 0$$\mod 6$ and $\equiv 2$$\mod 6$ can be achieved as characters of independent sequences. Furthermore these can be achieved using basic sequences.
\end{lemma}
\begin{proof}
Consider a basic sequence $S(A)=(a_k)$ with a basis $B=(b_k)$. Suppose that if $b_i=3^i$ for $i\ge k+1$, $b_{k+1}> \sum_{i=0}^{k}b_i$ and $(b_i)$ is a sequence which forms a basis, then it follows that $a_{2^{k+2}-1}=\sum_{i=0}^{k+1}b_i$ and that $a_{2^{k+2}}=b_{k+2}$. The character of any independent sequence is $\lambda=2a_{2^{k+1}-1}-a_{2^{k+1}}+1$, and using the previous definitions, it follows that \[\lambda=2\bigg(\sum_{i=0}^{k+1} b_i\bigg)-b_{k+2}+1=2\bigg(\sum_{i=0}^{k}b_i\bigg)-3^{k+1}+1\]\[=2\bigg(\sum_{i=0}^{k}b_i-3^i\bigg).\]
\\
Now in order for $(b_i)$ to satisfy the requirement of Theorem 1.9 we require $b_i=3^i\cdot \ell$ with $3$ not dividing $\ell$ for $0\le i\le k$. For the remainder of the argument, the generating function of a particular expression is the ordinary generating function of the set of nonnegative integer values the expression takes including multiplicity. It follows that $2(b_i-3^i)$ has a generating function of the form 
\[\sum_{j=0}^{\infty}x^{6j\cdot 3^i}+x^{(6j+2)\cdot 3^i}=\frac{1}{1-x^{2\cdot 3^i}}-\frac{x^{2\cdot 3^i}}{1-x^{2\cdot 3^{i+1}}}\]\[=\frac{1+x^{2\cdot 3^i}}{1-x^{2\cdot 3^{i+1}}}.\] Therefore the generating function of 
$2\bigg(\sum_{i=0}^{k}b_i-3^i\bigg)$ is \[\prod_{\ell=0}^{k}\frac{1+x^{2\cdot 3^\ell}}{1-x^{2\cdot 3^{\ell+1}}}\]\[=\bigg(\prod_{\ell=1}^{k-1}\frac{1+x^{2\cdot 3^\ell}}{1-x^{2\cdot 3^{\ell}}}\bigg)\bigg(\frac{1+x^2}{1-x^{2\cdot 3^{k+1}}}\bigg).\] Note that each of the terms in the product has a Taylor series with positive coefficients, so we can ``decrease" the coefficients of each individual term in the product and not introduce new positive coefficients. In particular note that \[\frac{1+x^{2\cdot 3^\ell}}{1-x^{2\cdot 3^{\ell}}}=1+2\sum_{k=1}^{\infty}x^{2k\cdot 3^\ell}\] and this contains the terms of the expansion $1+x^{2\cdot 3^\ell}+x^{4\cdot 3^\ell}$ while 
\[\frac{1}{1-x^{2\cdot 3^{\ell}}}=\sum_{k=0}^{\infty}x^{2k\cdot 3^\ell}\] 
contains $1+x^{2\cdot 3^k}+x^{4\cdot 3^k}$. Thus we can look at the product \[(1+x^2)\bigg(\prod_{\ell=1}^{k}1+x^{2\cdot 3^\ell}+x^{4\cdot 3^\ell}\bigg).\] But this is simply the the generating function for the sequence of numbers of the form $2(i_k i_{k-1}\cdots i_1 i_0)_3$ where the $i_j$ are digits in base $3$ and the only restriction is $0\le i_0\le 1$. Now the condition that $b_{k+1}> \sum_{i=0}^{k}b_i$ is equivalent to
$2\cdot 3^{k+1}> \sum_{i=0}^{k}2\cdot b_i$. Rearranging, it follows that $2\cdot 3^{k+1}-2\sum_{i=0}^{k}3^i\ge \sum_{i=0}^{k}2\big(b_i-3^i\big)$ or $3^{k+1}+1\ge \sum_{i=0}^{k}2\big(b_i-3^i\big)$. Thus, all characters of the above form $2(i_k i_{k-1}\cdots i_1 i_0)_3$ with $i_k=0$ and $0\le i_0\le 1$ give valid basic sequences. Taking $k$ sufficiently large, it follows that all numbers of the form $2(3\ell+1)$ or $2(3\ell)$, can be achieved as characters of an independent Stanley sequence ,and thus the desired result follows. 
\end{proof}
To construct the remaining characters claimed in Theorem 2.3 we construct a series of near-modular sets. Define the sets $A_i^{j}$ and $B_{i}^{j}$ for $1\le i\le 4$ as \[A_1^{j}=\{0,2,5,6+9j\}\] \[B_1^{j}=\{0,4,10,12+9j\}\] 
\[A_2^{j}=\{0,1,4,6,10,13,15,18+27j\}\] \[B_2^{j}=\{0,2,8,12,20,26,30,36+27j\}\]\[A_3^{j}=\{0,2,3,5,11,14,18,21,29,30,32,38,41,45,48,54+81j\}\]\[B_3^{j}=\{0,4,6,10,22,28,36,42,\\58,60,64,76,82,90,96,108+81j\}\] \[A_4^{j}=\{0,2,8,9,15,20,24,26,54,56,62,63,69,74,78,80,83,89,90,96,101,105,\]\[107,135,137,143,144,150,155,159,161,162+243j\}\]\[B_4^{j}=\{0,4,16,18,30,40,48,52,108,112,124,126,138,148,156,160,166,178,180,192,202,210,\]\[214,270,274,286,288,300,310,318,322,324+243j\}\] In particular, $A_{i}^{j}$ and $B_{i}^{j}$ are near-modular sets $\mod 3^{i+1}$ for $1\le i\le 4$. To demonstrate this note that $A_{i}^{0}$ and $B_{i}^{0}$ are near-modular sets via a computer verification and since we are simply increasing the largest element by a multiple of $3^{i+1}$ it is clear that  $A_{i}^{j}$ and $B_{i}^{j}$ are near-modular sets as well. (That $A_{i}^{0}$ and $B_{i}^{0}$ are near-modular sets verified in an auxiliary Java file.)
\begin{lemma}
All nonnegative integers $\lambda$ $\equiv 4$$\mod 6$ and $\lambda\not\equiv 244 \mod 486$ can be achieved as characters of independent sequences.
\end{lemma}
\begin{proof}
First note that this lemma in combination with the previous lemma proves Theorem 1.5. Now note that the largest element in $A_{i}^{j}$ is $2\cdot 3^i+3^{i+1}\cdot j$ and $|A_{i}^{j}|=2^{i+1}$. Suppose that $b_n=3^{n+i+1}$ for $n\ge k+1$, $3^{n+i+1}$ but not $3^{n+i+2}$ divides $b_n$ for all $n$, and $b_{k+1}> \big(\sum_{m=0}^{k}b_m\big)+\max\{A_{i}^{j}\}$. The proof of Theorem 2.3 implies that $a_{2^{k+3+i}-1}=2\cdot 3^i+3^{i+1}\cdot j+\sum_{m=0}^{k+1}b_m$ while $a_{2^{k+3+i}}=b_{k+2}$. The desired character $\lambda$ is therefore
\[\lambda=2a_{2^{k+3+i}-1}-a_{2^{k+3+i}}+1\]
\[=2\bigg(\max\{A_{i}^{j}\}+\sum_{j=0}^{k+1}b_j-\sum_{j=0}^{i+k+2}3^j\bigg)\]
\[=2\bigg(\max\{A_{i}^{j}\}-\sum_{j=0}^{i}3^j+\sum_{j=0}^{k+1}\Big(b_j-3^{j+i+1}\Big)\bigg)\]
\[=2j\cdot 3^{i+1}+2\bigg(\sum_{j=0}^{k+1}\big(b_j-3^{j+i+1}\big)\bigg)+3^i+1\]
\[=2j\cdot 3^{i+1}+2\bigg(\sum_{j=0}^{k}\big(b_j-3^{j+i+1}\big)\bigg)+3^i+1.\]
The last step follows from the assumption that $b_{k+1}=3^{k+i+2}$. It follows that $2(b_j-3^{j+i+1})$ has a generating function of the form 
\[\sum_{\ell=0}^{\infty}x^{6\ell\cdot 3^{j+i+1}}+x^{(6\ell+2)\cdot 3^{j+i+1}}\]\[=\frac{1+x^{2\cdot 3^{j+i+1}}}{1-x^{2\cdot 3^{j+i+2}}}.\] Therefore the given character has generating function of the form
\[x^{3^i+1}\bigg(\frac{1}{1-x^{2\cdot 3^{i+1}}}\bigg)\prod_{j=0}^{k}\frac{1+x^{2\cdot 3^{j+i+1}}}{1-x^{2\cdot 3^{j+i+2}}}.\] Rearranging this expression, we obtain
\[x^{3^i+1}\bigg(\frac{1}{1-x^{2\cdot 3^{k+i+2}}}\bigg)\prod_{j=0}^{k}\frac{1+x^{2\cdot 3^{j+i+1}}}{1-x^{2\cdot 3^{j+i+1}}}.\] Replacing the second factor with $1$ and $\frac{1+x^{2\cdot 3^{j+i+1}}}{1-x^{2\cdot 3^{j+i+1}}}$ with 
$1+x^{2\cdot 3^{j+i+1}}+x^{4\cdot 3^{j+i+1}}$ we obtain the simplified generating function
 \[x^{3^i+1}\prod_{j=0}^{k}\Big(1+x^{2\cdot 3^{j+i+1}}+x^{4\cdot 3^{j+i+1}}\Big).\] Taking all $k$ sufficiently large, all numbers of the form $3^i+1+2\ell(3^{i+1})$ can be achieved for $1\le i\le 4$. (The condition $b_{k+1}> \sum_{i=0}^{k}b_i+\max\{A_{i}^{j}\}$ can be met by taking only numbers with at most $k+i$ digits.) Preforming the same procedure for $B_i^{j}$, all numbers of the form $5\cdot 3^{i}+1+2\ell(3^{i+1})$ can be achieved as characters of an independent Stanley sequence. The desired result follows by using $A_{i}^{j}$ and $B_{i}^{j}$ constructed earlier for $1\le i\le 4$.
\end{proof}
\section{Conclusion}
This paper is temptingly close to proving that all even characters can be achieved as characters of independent Stanley sequences. To demonstrate all even characters can be achieved it suffices to prove the following conjecture.
\begin{conjecture}
For all $\ell\in \mathbb{Z}^{+}$, there exists a near-modular set $A$$\mod 3^{\ell+1}$ with $|A|=2^{\ell+1}$, and largest element of $A$ being $2\cdot 3^{\ell}$ or $4\cdot 3^{\ell}$.
\end{conjecture}
The sets $A_{i}^{0}$ and $B_{i}^{0}$ prove the conjecture for $\ell\in\{1,2,3,4\}$ by explicit construction. Finally, given the approach in this paper, the following question arises naturally.
\begin{question}
What characters can be achieved as characters of basic sequences?
\end{question}
Note that it is possible to achieve odd characters under certain bases, one example is given by $S(0,1,7)$ which has basis $\{1,7,10,30,90,270,\ldots\}$ and has character $7$. However which odd characters can be achieved is not immediately clear given the approach in this paper.
\section*{Acknowledgments}
The research was conducted at the University of Minnesota Duluth REU and was supported by NSF grant 1659047. The author would like to thank Joe Gallian for suggesting the topic of Stanley sequences, Colin Defant and Joe Gallian for reading over the manuscript with infinite patience, and to Richard Moy and David Rolnick for helpful conversations.

\bibliography{Character Values}
\bibliographystyle{plain}
\end{document}